\documentclass{amsart}

\newtheorem{theorem}{Theorem}[section]
\newtheorem{lemma}[theorem]{Lemma}
\newtheorem{corollary}[theorem]{Corollary}
\usepackage{graphics}
\usepackage{graphicx}
\numberwithin{equation}{section}

\DeclareMathOperator{\divv}{div}
\DeclareMathOperator{\Vol}{Vol}
\DeclareMathOperator{\Area}{Area}

\begin{document}

\title{A simple proof for Bernstein type theorems in Gauss space}

\author{Doan The Hieu}
\address{Department of Mathematics, College of Education, Hue University, Vietnam}
\curraddr{Department of Mathematics, College of Education, Hue University, Vietnam}
\email{dthehieu@yahoo.com}


\date{}



\begin{abstract}
A weighted area estimate for  entire graphs with bounded weighted mean curvature in Gauss space is given by a simple proof. Bernstein type theorems for self shrinkers (\cite {wa}) as well as for graphic $\lambda$-hypersurfaces (\cite{ chwe2}) follow immediately as consequences. 
\end{abstract}

\maketitle
\section{Introduction}
A manifold with density is a Riemannian manifold with a positive function $e^{-f}$ used to weight both volume and perimeter area. The weighted volume of a region $E$ is $\Vol_f(E)=\int_Ee^{-f}dV$ and the weighted area of a hypersurface $\Sigma$ is $\Area_f(\Sigma)=\int_{\Sigma}e^{-f}dA_{\Sigma},$ where $dV$ and $dA$ are the $(n+1)$-dimensional Riemannian volume and $n$-dimensional Riemannian perimeter area elements, respectively. 

The weighted mean curvature of a hypersurface $\Sigma$ on such a manifold
 is defined as follows
$$ \label {Hf}
    H_{f}(\Sigma)=H(\Sigma)+\langle \nabla f,{\bf n}\rangle,$$
where ${\bf n}$ is the unit
normal vector field and $H=-\divv {\bf n}$ is the Euclidean mean curvature of the hypersurface.  If $H_f(\Sigma)=\lambda,$ a constant, then $\Sigma$ is called a $\lambda$-hypersurface and if $H_f(\Sigma)=0,$ then $\Sigma$ is said to be $f$-minimal.

Gauss space $\mathbb G^{n+1},$ Euclidean space $\mathbb R^{n+1}$ with Gaussian probability density $e^{-f}=(2\pi)^{-\frac {n+1}2}e^{-\frac{|x|^2}2},$ is a typical example of a manifold with density and very interesting to probabilists. For more details about manifolds with density, we refer the reader to \cite{mo1}, \cite{mo2}, \cite{mo5}, \cite{muwa}.

In Gauss space, $f$-minimal hypersurfaces are self-shrinkers  and hyperplanes are $\lambda$-hypersurfaces. The weighted mean curvature of the hyperplane $\sum_{i=1}^{n+1}a_ix_i+a_0=0$ is $\frac {-a_0}{(\sum_{i=1}^{n+1}+a_i^2)^{1/2}}.$
It is well-known that, hyperplanes solve the weighted isoperimetric problem, i.e.,
 they minimize weighted area for given weighted volume (see \cite{bo}, \cite{suts}).

In this short paper, by a simple proof, we give a weighted area estimate for entire graphs with bounded weighted mean curvature. The Bernstein type theorems for graphic self-shrinkers (\cite{wa}) as well as for graphic $\lambda$-hypersurfaces (\cite{chwe2}) follow immediately as consequences. These results were also proved by Ecker and Huisken \cite{echu} and Guang \cite {gu}, respectively, under the polynomial volume growth conditions.
\section{Weighted area estimate and Bernstein type theorems}

In $\mathbb G^{n+1},$ let $\Sigma$ be the graph of a smooth function $u(\textbf x)=x_{n+1}, \textbf x\in \mathbb R^n$ and  ${\bf n} $ be its  upward unit normal field. Extending ${\bf n}$ by translations  along $x_{n+1}$-axis we obtain
 a smooth vector field on $\mathbb R^{n+1},$ also denoted by ${\bf n}.$  Along any vertical line, since $\divv({\bf n})$ is unchange while $\langle\nabla f, {\bf n}\rangle$ is increasing,  $H_f=-\divv({\bf n})+\langle\nabla f, {\bf n}\rangle$ is increasing.
 
Consider the $n$-differential form
$$ w(X_1, X_2,\ldots, X_{n})= \det(X_1, X_2,\ldots, X_{n}, {\bf n}),$$
 where $X_i,\ i=1,2,\ldots, n$ are smooth vector fields.
 It follows that
  $|w(X_1, X_2,\ldots, X_{n})|\le 1,$ for every unit normal vector fields  $X_i,\ i=1,2,\ldots, n$ and the equality holds if and only if
       $X_1, X_2,\ldots, X_{n}$ are tangent to $\Sigma.$

Denote by:
\begin{itemize}
\item $E_1=\{(\textbf x, x_{n+1}) : x_{n+1}\le u(\textbf x)\},\ E_2=\{(\textbf x, x_{n+1}) : x_{n+1}\le a\},\ a\in \mathbb R,$ such that $\Vol_f(E_1)=\Vol_f(E_2)$ and $P$  the hyperplane $x_{n+1}=a;$
\item $F=(E_1-E_2)\cup (E_2-E_1),$  the region bounded by $P$ and $\Sigma;$ 
 \item $F^+=E_2-E_1$ and $F^-=E_1-E_2,$ the parts of $F$ above  and under $\Sigma,$ respectively;
 \item  $B_R, S_R$ the $(n+1)$-ball and $n$-hypershere in $\mathbb R^{n+1}$ with center $O$ and radius $R,$ respectively; 
 \item $\Sigma_R=\Sigma \cap B_R,\ P_R=P\cap B_R,\ F_R=F\cap B_R,\  F^+_R=F^+\cap B_R$ and $F^-_R=F^-\cap B_R.$
     \end{itemize}

\begin{lemma}
If $H_f(\Sigma)$ is bounded, then 
\begin{equation}\label{eq3}
\Area_f (\Sigma)\le \Area_f(P)+\frac 12(M-m)\Vol_f(F),
\end{equation}
where $M=\sup H_f(\Sigma)$ and $ m=\inf H_f(\Sigma).$ 
\end{lemma} 
\begin{proof} 
Let $R$ be large enough, such that $B_R$ intersects both $F^+$ and $F^-.$  By  Stokes' theorem together by choosing suitable orientations for objects (see the picture), we have
$$\begin{aligned} 
\Area_f (\Sigma_R)-\Area_f(P_R) +\int_{ F\cap S_R}e^{-f} w&\le \int_{\Sigma_R}e^{-f} w-\int_{P_R}e^{-f} w+\int_{F\cap S_R}e^{-f} w\\
&=\int_{F_R}d(e^{-f} w)=\int_{F_R}d(\divv (e^{-f}{\bf n}))dV\\
 &=\int_{F_R}( e^{-f}\divv({\bf n})-e^{-f}\langle\nabla f, {\bf n}\rangle) dV\\
 &=-\int_{F_R}e^{-f}H_f dV\\
&=-\int_{F_R^+}e^{-f}H_f dV+\int_{F_R^-}e^{-f}H_f dV\\
&\le -m\Vol_f(F_R^+)+M\Vol_f(F_R^-).
\end{aligned}$$
	Thus,
\begin{equation}\label{eq4} \Area_f (\Sigma_R)-\Area_f(P_R) +\int_{ F_R}e^{-f} w\le -m\Vol_f(F_R^+)+M\Vol_f(F_R^-).
\end{equation}
 \includegraphics[scale=.4]{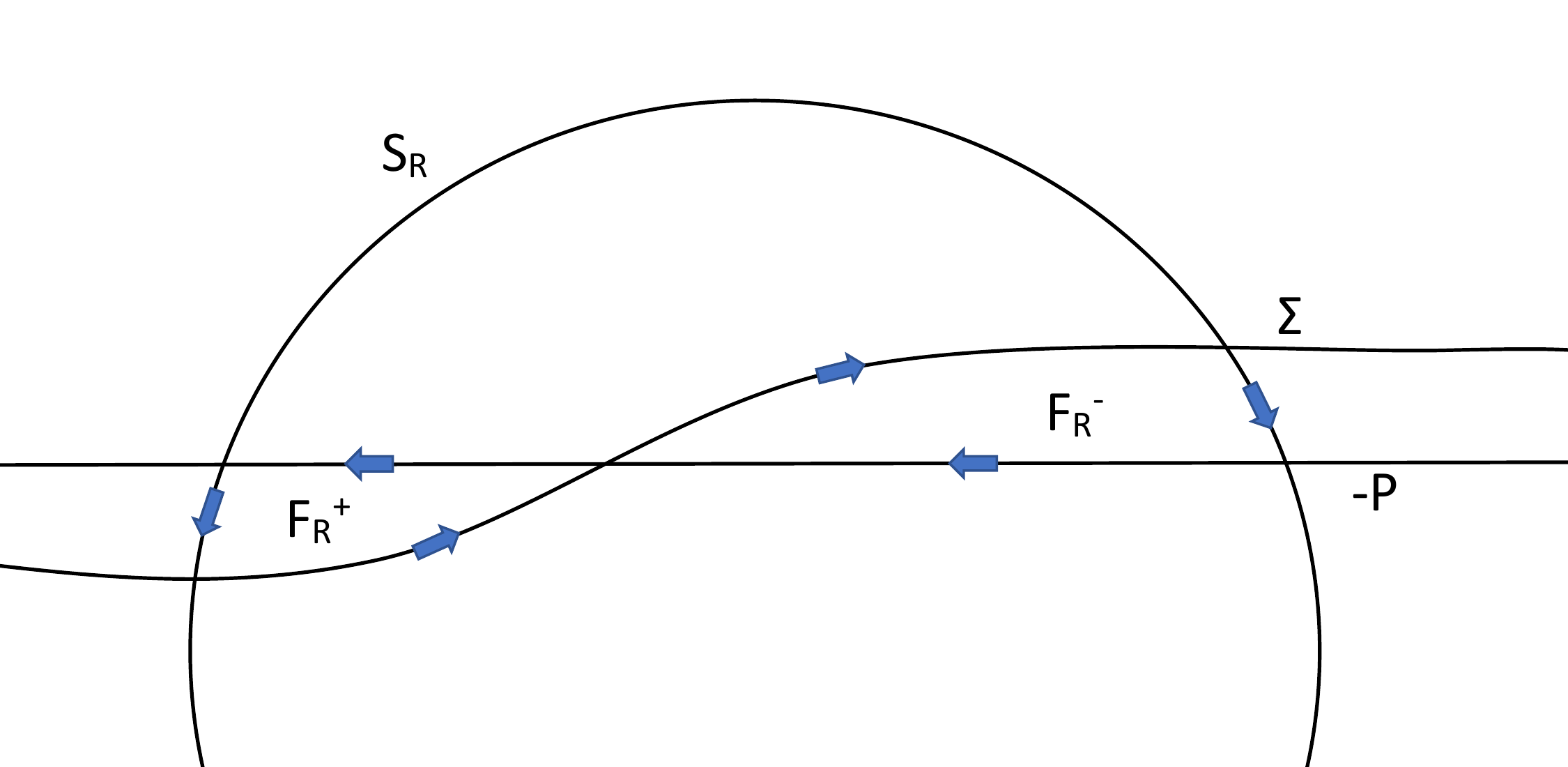}\\
 \vskip.5cm
It is not hard to check that
$$\lim_{R\rightarrow\infty}\int_{S_R\cap F}e^{-f} w=\lim_{R\rightarrow\infty}e^{-R}\int_{S_R\cap F}w=0,$$
and by the assumption that $\Vol_f(E_1)=\Vol_f(E_2),$ 
$$\lim_{R\rightarrow\infty}\Vol_f(F_R^+)=\Vol_f(F^+)=\lim_{R\rightarrow\infty}\Vol_f(F_R^-)=\Vol_f(F^-)=\frac 12\Vol_f(F).$$ 
Taking the limit of both sides of (\ref{eq4}) as $R$ goes to infinity, we get
(\ref{eq3}).\end{proof}
\begin{corollary} [Bernstein type theorem for $\lambda$-hypersurfaces \cite{chwe2}]
If $\Sigma$ is an entire graphic $\lambda$-hypersurface, then it must be a hyperplane.
\end{corollary}
\begin{proof}
Because $M-m$=0 and $P$ is weighted area minimizing.  \end{proof}
\begin{corollary} [Bernstein type theorem for self-shrinkers \cite{wa}]
If $\Sigma$ is an entire graphic self-shrinker, then it must be a hyperplane passing through the origin..
\end{corollary}
\begin{proof} Because among all hyperplanes, only the ones passing the origin have zero weighted mean curvature.
  \end{proof}

\bibliographystyle{amsplain}

\end{document}